\documentclass[article]{amsart}
\usepackage[T1]{fontenc}
\usepackage{amsthm}
\usepackage{amssymb}
\usepackage{amsmath}
\usepackage{mathabx}
\usepackage{xcolor}
\usepackage{verbatim, color}
\usepackage{bbm}
\usepackage{float}
\usepackage{dsfont, graphicx}
\usepackage{epstopdf}
\usepackage[hidelinks]{hyperref}
\usepackage[T1]{fontenc}
\usepackage{amssymb,amsmath, amsthm, amsfonts}
\usepackage{graphicx}

\usepackage[hidelinks]{hyperref}
\usepackage{pagerange}
\usepackage{thmtools,thm-restate}

\makeatletter
\numberwithin{equation}{section}
\numberwithin{figure}{section}
\theoremstyle{plain}
\newtheorem{thm}{Theorem}[section]
\newtheorem{prop}[thm]{Proposition}
\newtheorem{definition}[thm]{Definition}

\newtheorem{lem}[thm]{Lemma}
\newtheorem{cor}[thm]{Corollary}
\newtheorem{rem}[thm]{Remark}

  \newcounter{casectr}
  
\theoremstyle{definition}

\theoremstyle{remark}

\newcommand{\LLL}{\mathcal{L}}

\begin{document}
\address{Chenjie Fan
\newline \indent Academy of Mathematics and Systems Science and Hua Loo-Keng Key Laboratory of \indent Mathematics, Chinese Academy of Sciences,\indent 
\newline \indent  
Beijing, China.\indent }
\email{cjfanpku@gmail.com}

\address{Gigliola Staffilani
\newline \indent 
Department of Mathematics, Massachusetts Institute of Technology, \indent 
\newline \indent
Cambridge, MA 02139, USA. \indent }
\email{gigliola@math.mit.edu}

\address{Zehua Zhao
\newline \indent Department of Mathematics and Statistics, Beijing Institute of Technology, Beijing, China.
\newline \indent MIIT Key Laboratory of Mathematical Theory and Computation in Information Security,\indent 
\newline \indent Beijing, China. \indent}
\email{zzh@bit.edu.cn}

\title{On decaying properties of nonlinear Schr\"odinger equations}
\author{Chenjie Fan}  
%\address{}
%\email{} 
%\thanks{}

\author{Gigliola Staffilani}
%\address{Department of Mathematics, Massachusetts Institute of Technology, Cambridge, MA 02139, USA}
%\email{gigliola@math.mit.edu} 
%\thanks{G.S. is  funded in part by  the NSF grants DMS-1764403, DMS-2052651 and the Simons Foundation through the Simons Collaboration on Wave Turbulence.}

\author{Zehua Zhao}
%\address{}
%\email{} 
%\thanks{}

\maketitle

\begin{abstract}
In this paper we discuss  quantitative (pointwise) decay estimates for solutions to the 3D cubic defocusing Nonlinear Schr\"odinger equation with various  initial data, deterministic and random. We  show  that nonlinear solutions enjoy the same decay rate as the linear ones. The regularity assumption on the initial data is much lower than in previous results (see \cite{fan2021decay} and the references therein) and moreover  we  quantify the  decay, which is another novelty  of this work. Furthermore, we show that the  (physical) randomization of the initial data can be used to replace the $L^1$-data assumption (see \cite{fan2022note} for the necessity of the $L^1$-data assumption). At last, we note that this method can be also applied to derive decay estimates for  other nonlinear dispersive equations.
\end{abstract}

\tableofcontents
\section{Introduction}
\subsection{Background and motivations}
Linear dispersive estimates on unbounded  domains play a fundamental role in the study of nonlinear dispersive PDEs. It is in some sense  the starting point for studying local well-posedness of nonlinear problems. In the Schr\"odinger case, the dispersive estimate in $\mathbb{R}^d$  reads as
\begin{equation}\label{eq: dispersiveclassical}
\|e^{it\Delta}f\|_{L_{x}^{\infty}(\mathbb{R}^d)}\lesssim |t |^{-\frac{d}{2}}\|f\|_{L^1_x(\mathbb{R}^d)},
\end{equation}
where $d$ is the spatial dimension.

 For the nonlinear defocusing Schr\"odinger (NLS) equation, great progress has been made in recent years to understand the scattering behavior of solutions,  see for example, \cite{bourgain1999global}, \cite{colliander2008global}, \cite{dodson2012global}, \cite{kenig2006global}. Those results say that  given an initial data in a certain (critical) Sobolev space $\dot{H}^{s}$  there exists a unique   global  solutions $u$ to NLS, and the solution scatters in the sense that  there exists some $u^{\pm}$, so that 
\begin{equation}\label{eq: scatterclassical}
	\|u(t)-e^{it\Delta}u^{\pm}\|_{\dot{H}^{s_{c}}(\mathbb{R}^d)}\rightarrow 0, \quad \textmd{ as } t\rightarrow \pm \infty.
\end{equation}
From \eqref{eq: dispersiveclassical} and \eqref{eq: scatterclassical} one may then conjecture for the nonlinear solutions $u$ a similar estimate  as \eqref{eq: dispersiveclassical}, for example for $d=3$
\begin{equation}\label{eq: decayo}
\|u(t)\|_{L_{x}^{\infty}(\mathbb{R}^3)}\leq C(u_{0})|t|^{-\frac{3}{2}}.	
\end{equation}

 Note that conclusion \eqref{eq: decayo}  is not a-priori obvious since for example it is not known whether $u^{\pm}$ in \eqref{eq: scatterclassical} is in $L^{1}$, and moreover the convergence in \eqref{eq: scatterclassical} is not a-priori in $L^{\infty}$ and certainly the rate is also not known. Indeed, the study of decaying estimate for nonlinear solutions is related to the asymptotic convergence rate in \eqref{eq: scatterclassical} (see the Appendix of \cite{fan2022note}).

 The decay estimates for nonlinear dispersive equations have been studied widely. In \cite{lin1978decay}, Lin and Strauss studied the decay of the $L^{\infty}$-norm of solutions to the 3D NLS based on Morawetz estimate. (See also Corollary 3.4 in \cite{grillakis2013pair} for the cubic Hartree case.) It is also possible to apply the vector field methods and use commutator type estimates to derive decay estimates, see for example, \cite{klainerman1985uniform}. (See also \cite{klainerman1983global}, \cite{shatah1982global}, \cite{hayashi1986c}.) 
  
  For the Schr\"odinger case, in particular, before the work in \cite{bourgain1999global},  decay estimates \eqref{eq: decayo} were key steps to prove  scattering results. We also refer to \cite{visciglia2008decay}, and the references therein, for results on decay estimates regarding NLS without decay rates. 
  
  In the present article, as in  previous works of the first and the third authors, \cite{fan2021decay}, \cite{fan2022note}, the starting point is whether,  given given the fact that scattering behaviors have by now been studied   extensively, can one further improve the  understanding of \eqref{eq: decayo}? Conceptually, one  wants to understand how fast or slow, 
  scattering  behaviors can happen. The answer to this question  will give more quantitative estimates for $C(u_{0})$ in \eqref{eq: decayo} as well.

 To make the question more concrete, let us focus on the defocusing cubic NLS in 3D, 
 \begin{equation}\label{eq: dcub}
 iu_{t}+\Delta u=|u|^{2}u, \quad u(0,x) = u_0(x).	
 \end{equation}
It has been proved in \cite{fan2021decay}, \cite{fan2022note}, for all $u_{0}\in H^{4}\cap L^{1}$,  one has that \eqref{eq: decayo} holds and  $C(u_{0})$  only depends on the size $\|u_{0}\|_{H^{4}\cap L^{1}}$ rather than the profile of $u_{0}$.
This follows from some concentration compactness consideration. But it was not clear how this $C$ depends on the $\|u_{0}\|_{_{H^{4}\cap L^{1}}}$ in a quantitative way. See also \cite{guodecay}.

It has also been proved in \cite{fan2022note} that not even a weaker version of \eqref{eq: mainestimate} can  hold if one replaces $X$ with only a  Sobolev type  space $H^{s}$. To be more precise, for any $g(t)>0$ that goes to infinity as $t\rightarrow \infty$, we can construct for instance $H^{100}$  solutions to \eqref{eq: dcub} such that 
\begin{equation}
\limsup_{t\rightarrow \infty}g(t)\|u(t)-e^{it\Delta}u^{+}\|_{\dot{H}^{1/2}}=\infty.
\end{equation}

Compared to  \cite{fan2021decay}, \cite{fan2022note} and other previous results, to the best of our knowledge, the main three \textbf{new} points in the current  paper are: 1. To obtain an estimate such \eqref{eq: mainestimate}  we can  considerably \textbf{lower} the regularity assumption of the solution, to almost the critical level (see Theorem \ref{thm3}); 2. We can obtain the \textbf{quantitative} decay results by characterizing the implicit constant in the decay estimates (see Theorem \ref{thm1}, Theorem \ref{thm2}, Theorem \ref{thm3} and Theorem \ref{thm4}) ; 3. We discuss the \textbf{randomized} case and show that  (physical) randomization of the initial data can be used to replace the $L^1$-data assumption (see Theorem \ref{thm4} and see \cite{fan2022note} for the necessity of the $L^1$-data assumption). 

It is natural to ask why the $L^{1}$ assumption and randomization in physical space would be involved, and in some sense necessary, in the quantitative study of decay estimates for the NLS.
To see this, recall that for every scattering solution $u$ to \eqref{eq: dcub}, for any given $\delta>0$, one can find $L>0$, so that 
 \begin{equation}\label{eq: smalll}
 \|u(t)\|_{{L^5_t}[L,\infty)L_{x}^{5}}<\delta.	
 \end{equation}
 However, it is impossible to characterize this $L$ in any quantitative way if one considers initial data $u_{0}\in H^{s}$. Since, for any given $u_{0}$, one can evolve backward, (non-linearly), for a long time, and get a new initial data, which delays the $L$ so that \eqref{eq: smalll} holds. To be more precise, let us fix a Schwarz initial data $u_{0}$, with $\|u_{0}\|_{H^{s}}\leq 1$. Let $u$ be the associated solution. We know that $\|u\|_{L_{t,x}^{5}[0,\infty)}\geq \delta_{0}>0$ for some $\delta_{0}$, and $\sup_{t}\|u(t)\|_{H^{s}}\leq M$ for some $M>0$, since such a solution scatters. Let $u_{0,n}(x):=u(-n,x)$ and let $u_{n}$ be the associated solution with initial data $u_{0,n}$. Note that $\|u_{0,n}\|_{H^{s}}\leq M$. Now  fix $\delta=\delta_{0}/2$ and evaluate $L_{n}$ so that \eqref{eq: smalll} holds. One has $L_{n}\geq n$, and cannot be bounded by $M$. The problem lies in the time translation symmetry in the cubic NLS. Both $L^{1}$ assumption and randomization in physical space, play a role of removing the time translation symmetry in cubic NLS.
 
We believe that the methods in this paper can be applied to derive decay estimates for  other dispersive models with suitable modifications. See Section \ref{6} for more discussions.\vspace{3mm}

\subsection{Notations}
Throughout this note, we use $C$ to denote  the universal constant and $C$ may change line by line. Also, $\alpha,\beta$ may change line by line. We say $A\lesssim B$, if $A\leq CB$. We say $A\sim B$ if $A\lesssim B$ and $B\lesssim A$. We also use notation $C_{B}$ to denote a constant depends on $B$. We use usual $L^{p}$ spaces and Sobolev spaces $H^{s}$. Since we will always work on $\mathbb{R}^{3}$, we will write $L^{p}(\mathbb{R}^{3})$ as $L^{p}$ and $L_{t,x}^{p}(I\times \mathbb{R}^{3})$ as $L_{t,x}^{p}(I)$. We will also use $S^{s}$ to denote the Strichartz space,
\begin{equation}
\|u\|_{S^{s}(I)}:=\|\langle D \rangle ^{s}u\|_{L_{t}^{q}L_{x}^{r}(I)}.	
\end{equation}

\subsection{The statement of main results}
We are now ready to state our results.
Consider the Cauchy problem for the 3D cubic defocusing NLS,
\begin{equation}\label{eq: cubicnls}
iu_{t}+\Delta u=|u|^{2}u,\quad
u(0,x)=u_{0}(x).	
\end{equation}
The purpose of the current article is to present decay estimates of the form
\begin{equation}\label{eq: mainestimate}
\|u(t)\|_{L_{x}^{\infty}}\leq C(\|u_{0}\|_{X})|t|^{-\frac{3}{2}}.
\end{equation}
The constant $C$ in \eqref{eq: mainestimate} in this article will only depend on the size (but not the profile) of the initial data, measured by a certain norm $\|\cdot\|_{X}$. We will characterize how this $C$ depends on $\|u_{0}\|_{X}$. With different choices of $\|\cdot\|$ detailed below, we will derive results involving polynomial dependence or exponential dependence.

Below, $\alpha, \beta, C$ are  constants that may change line by line. We have
\begin{thm}\label{thm1}
For the  initial value problem \eqref{eq: cubicnls} with initial data $u_{0}\in X=H^{1}\cap L^{1}$, one has that \eqref{eq: mainestimate} holds with
\begin{equation}\label{eq: me1}
C(\|u_{0}\|_{X})=C\exp C\|u_{0}\|_{X}^{\beta},	
\end{equation}
for some $\beta>0$.
\end{thm}
\begin{cor}\label{cor: thm1}
We also note that, if one considers the asymptotic behavior, the implicit constant in Theorem \ref{thm1} can be improved from the exponential bound to the polynomial bound in the following sense,
\begin{equation}
\limsup_{t\rightarrow \infty} t^{\frac{3}{2}}\|u(t)\|_{L_{x}^{\infty}}\leq C(1+\|u_{0}\|_{X})^{\beta}, 
\end{equation}
for some $\beta>0$. We will give the proof for this remark in the end of Section 3.1 (after giving the proof for Theorem \ref{thm1}). We note that the analogous statement also holds for \eqref{eq: thm1ran} in Theorem \ref{thm4}.
\end{cor}
\begin{rem}
We note that the quantitative decay result Theorem \ref{thm1} immediately gives the \textbf{quantitative scattering rate} via direct integration and Strichartz estimate, which describes how fast the solution scatters to the final states. (See the Appendix of \cite{fan2022note} for more details). To be more precise, one can obtain the following: for $u$ satisfies \eqref{eq: cubicnls} with initial data $u_{0}\in X=H^{1}\cap L^{1}$
\begin{equation}
  \|u(t)-e^{it\Delta}u^{+}\|_{\dot{H}_{x}^{\frac{1}{2}}} \leq C(\|u_{0}\|_{X}) t^{-2},
\end{equation}
and
\begin{equation}\label{eq: rate2}
  \|u(t)\|_{L_{t,x}^{5}(t\geq s)}  \leq C(\|u_{0}\|_{X}) s^{-\frac{7}{10}},
\end{equation}
where $C(\|u_{0}\|_{X})=C\exp C\|u_{0}\|_{X}^{\beta}$.
\end{rem}
If one further assumes $xu_{0}\in L^{2}$ (i.e. the finite variance condition), we can improve the  exponential bound to the polynomial bound in the dispersive estimate \eqref{eq: mainestimate} in the following sense:
\begin{thm}\label{thm2}
	For the  initial value problem \eqref{eq: cubicnls} with initial data $u_{0}\in H^{1}\cap L^{1}$ and $xu_{0}\in L^{2}$, let $\|u_{0}\|_{X}=\|u_{0}\|_{H^{1}\cap L^{1}}+\|xu_{0}\|_{L^2}$ one has that \eqref{eq: mainestimate} holds for
\begin{equation}\label{eq: me2}
C(\|u_{0}\|_{X})=C(1+\|u_{0}\|_{X})^{\beta},	
\end{equation}
for some $\beta>0$.
\end{thm}

It is often natural to consider NLS in the $H^{1}$ space since it is corresponding to the energy conservation law. Sometimes, it is also of interest to lower the $H^{s}$ regularity of initial data, and we recall that the Schr\"odinger initial value problem \eqref{eq: cubicnls} is $H^{\frac{1}{2}}$ critical. We have the $H^{\frac{1}{2}}$-type result as follows.
\begin{thm}\label{thm3}
	For the  initial value problem \eqref{eq: cubicnls} with initial data $u_{0}\in H^{s}\cap L^{1}$ and $xu_{0}\in L^{2}$, let $\|u_{0}\|_{X}=\|u_{0}\|_{H^{s}\cap L^{1}}+\|xu_{0}\|_{L^2_x}$, where $s>\frac{1}{2}$,  one has that \eqref{eq: mainestimate} holds for
\begin{equation}\label{eq: me2}
C(\|u_{0}\|_{X})=C\exp (\|u_{0}\|_{X}^{\beta}),	
\end{equation}
for some $\beta>0$.
\end{thm}
\begin{rem}
We remark that, to reach the endpoint case of Theorem \ref{thm3}, i.e. to reach the  initial data $u_{0}\in H^{\frac{1}{2}}\cap L^{1}$ and $xu_{0}\in L^{2}$, would  be as hard as to reach the (quantitative) global well-posedness for \eqref{eq: cubicnls} with initial data in $H^{1/2}$, which is a  major open problem. See however Dodson's recent works, \cite{dodson2021global}, \cite{dodson2021scattering}. Indeed, it is also open to prove global well-posedness  of \eqref{eq: cubicnls} with $H^{s}, s>\frac{1}{2}$ initial data. We will briefly point out in the proof of Theorem \ref{thm3} why $xu_{0}\in L^{2}$ is  of help here, i.e. it is not hard to prove GWP for initial data such that $\|u_{0}\|_{H^{s}}+\|xu_{0}\|_{2}<\infty, s>1/2$.
\end{rem}

Sometimes, it is not favorable to have the $L^{1}$ condition. We remark  below that one can remove the $L^{1}$ assumption in Theorem \ref{thm1}-\ref{thm3} by performing  a randomization in physical space for the initial data. 
To be more precise, 
let $\phi_{n}(x):=\phi(x-n), n \in \mathbb{Z}^{3}$ be a partition  of unity of $\mathbb{R}^{3}$
\begin{equation}
1=\sum_{n}\phi_{n}=\sum_{n}\phi(x-n).
\end{equation}
 Let $g_{n}(\omega)$ be i.i.d. standard Gaussian, 
 and let
\begin{equation}
u^{\omega}_{0}(x)=\sum \phi_{n}(x)g_{n}(\omega)u_{0}(x).
\end{equation}
We recall  that randomization in physical space was also used in \cite{murphy2019random}, see also \cite{bringmann2021almost}.

We assume $\|u_{0}\|_{X}=1$, where $X=H^{1}$ or $H^{1}\cap \|xu_{0}\|_{L^2}$ or $H^{s}\cap \|xu_{0}\|_{L^2}$. Note that the size of data can be absorbed in the Gaussian. We then have
\begin{thm}\label{thm4}
Consider  the initial value problem \eqref{eq: cubicnls} with randomized initial data $u_{0}^{\omega}$. Let $u^{\omega}$ be the associated global solutions. Let $u^{\omega}_{nl}=u^{\omega}-e^{it\Delta}u_{0}^{\omega}$.
\begin{itemize}
\item If $\|u_{0}\|_{H^{1}}\leq 1$, then except for a  small probability set of size $\exp A^{-\alpha}$, one has 
\begin{equation}\label{eq: thm1ran}
\|u_{nl}^{\omega}\|_{L_{x}^{\infty}}\leq C(\exp A^{\beta})|t|^{-\frac{3}{2}},	
\end{equation}
for some $\alpha, \, \beta>0$.
\item If $\|u_{0}\|_{H^{1}}+\|xu_{0}\|_{L^2_x}\leq 1$, then except for a  small probability set of size $\exp A^{-\alpha}$, one has 
\begin{equation}\label{eq: thm2ran}
	\|u_{nl}^{\omega}\|_{L_{x}^{\infty}}\leq C(1+A)^{\beta}|t|^{-\frac{3}{2}},
\end{equation}
for some $\alpha, \, \beta>0$.
\item If $\|u_{0}\|_{H^{s}}+\|xu_{0}\|_{L^2}\leq 1$, 
$s>\frac{1}{2}$, then up to small probability $\exp A^{-\alpha}$, one has 
\begin{equation}\label{eq: thm3ran}
\|u_{nl}^{\omega}(t)\|_{L_{x}^{\infty}}\leq C(\exp A^{\beta})|t|^{-\frac{3}{2}},
\end{equation}
for some $\alpha, \, \beta>0$.
\end{itemize}

\end{thm}
\begin{rem}
Regarding the linear evolution of a  random data as defined above, though we cannot obtain the pointwise estimate for them,  they behave in time average like a linear evolution of $L^{1}$ data. Indeed, that is  why one can remove the $L^{1}$ assumptions by doing randomization in physical space. 
\end{rem}
  It is also possible to perform  randomization in frequency space  to lower the regularity assumption. This is a very active research field ever since the seminal work of Bourgain, \cite{bourgain1994periodic}, \cite{bourgain1996invariant} in the periodic setting, see also the recent breakthrough \cite{deng2022random} and references therein. We do not  discuss this issue here.

\subsection{A technical remark}
In this subsection, we briefly explain why we can lower the regularity assumption in this paper (see the proofs in the following sections for more details). For convenience, we compare Theorem \ref{thm1} ($H^1$-regularity assumption) in this paper with Theorem 1.4 ($H^4$-regularity assumption) in \cite{fan2021decay}.

As in \cite{fan2021decay}, we decompose the nonlinear solution $u$ into several parts and we want to control all of them since we intend to use a bootstrap argument to show the decay estimate. The most non-trivial term is:
\begin{equation}
F_3=i\int_{t-M}^t e^{i(t-s)\Delta}(|u|^{2}u)(s) ds,    
\end{equation}
where $M$ is a positive constant depending on the size of  the initial data. This term is part of the Duhamel expression  of the nonlinear solution when the integral is close to $t$.

We intend to control the $L^{\infty}_x$-norm of $F_3$. However, to control this term, higher regularity ($H^4$-regularity) is required in \cite{fan2021decay}. The reason is that, if one uses the dispersive estimate directly, one has
\begin{align}
  \big\| F_3  \big\|_{L_x^{\infty}} &= \big\| \int_{t-M}^t e^{i(t-s)\Delta}|u|^2uds  \big\|_{L_x^{\infty}} \\
  &\lesssim \int_{t-M}^t \big\|  e^{i(t-s)\Delta}|u|^2u \big\|_{L_x^{\infty}}ds \\
  &\lesssim \int_{t-M}^t   (t-s)^{-\frac{3}{2}} \big\| |u|^2u \big\|_{L_x^{1}}ds \\  
  &\lesssim \int_{t-M}^t   (t-s)^{-\frac{3}{2}} \big\| u \big\|^3_{L_x^{3}}ds. \\  
\end{align}
We note that, for the integral over $[t-M,t]$, the $(t-s)^{-\frac{3}{2}}$-term is too singular (non-integrable), thus it is not possible to control $\big\| F_3  \big\|_{L_x^{\infty}}$ by $ct^{-\frac{3}{2}}$ in this way, for some $c$ small. In order to solve this issue, one can control the $L^{\infty}_x$-norm of $F_3$ by estimating its $H^4$, $\dot{H}^1$ and $L^2$-norms (see the proof of Theorem 1.4 in \cite{fan2021decay} for details). However, doing this will inevitably cause higher-regularity requirement .

For the current paper, we now observe that, using a trick (Sobolev  and H\"older inequalities), we can control the $L^{\infty}_x$-norm of $F_3$ in the following way,
\begin{align*}
   \big\| F_3  \big\|_{L_x^{\infty}} &=\big\| \int_{t-M}^t e^{i(t-s)\Delta}|u|^2uds  \big\|_{L_x^{\infty}} \\
  &\lesssim \big\| \int_{t-M}^t e^{i(t-s)\Delta}|\nabla_x|(|u|^2u)ds  \big\|_{L_x^{3}} \\
  &\lesssim \int_{t-M}^t \big\|  e^{i(t-s)\Delta}|\nabla_x|(|u|^2u) \big\|_{L_x^{3}} ds \\
  &\lesssim \int_{t-M}^t (t-s)^{3(\frac{1}{2}-\frac{1}{3})}\big\|  |\nabla_x|(|u|^2u) \big\|_{L_x^{\frac{3}{2}}} ds \\
  &\lesssim \big( \int_{t-M}^t ((t-s)^{3(\frac{1}{2}-\frac{1}{3})})^{2-}ds \big)^{\frac{1}{2-}} \cdot \big( \int_{t-M}^t (\big\|  |\nabla_x|(|u|^2u) \big\|_{L_x^{\frac{3}{2}}})^{2+} ds \big)^{\frac{1}{2+}}.    
\end{align*}
We can see that, the advantage of performing  the Sobolev inequality first is to lower the spatial exponent for the integrand from $\infty$ to $3$ thus after using the dispersive estimate, we end up with $(t-s)^{-\frac{1}{2}}$ which is  not too singular near $t$ (i.e. it is integrable). Then, applying  H\"older inequality, it suffices to deal with the second term in the last line above, which is manageable. Thus, one can handle the $F_3$-term based on the scattering result. We refer to Section \ref{3}, Section \ref{4} and Section \ref{5}  for more proof details. 

This observation/trick is proved to be useful to lower the regularity requirement for many cases when one considers  nonlinear decay problems. We also note that similar ideas are useful for studying the long time dynamics for the stochastic NLS, see the recent works \cite{fan2020long,fan2021long} for more information.

\subsection{Structure of the paper}
The rest of the article is organized as follows. In Section \ref{pre}, we include the basic estimates and the global results for 3D cubic NLS; in Section \ref{3}, we give the proof for Theorem \ref{thm1} and Theorem \ref{thm2}; in Section \ref{4}, we give the proof for Theorem \ref{thm3}; in Section \ref{5}, we give the proof for Theorem \ref{thm4}; in Section \ref{6}, we give some further remarks on the applications of this method for other models.

\subsection*{Acknowledgment} C. F. was partially supported by the  National Key R\&D Program of China, 2021YFA1000800, CAS Project for Young Scientists in Basic Research, Grant No. YSBR-031, and NSFC grant No.11688101. G.S. is  funded in part by  the NSF grants DMS-1764403, DMS-2052651 and the Simons Foundation through the Simons Collaboration on Wave Turbulence. Z. Z. was supported by the NSF grant of China No. 12101046, 12271032,  and the Beijing Institute of Technology Research Fund Program for Young Scholars. 
\section{Preliminaries}\label{pre}
In this section, we collect some useful results and estimates, including the standard dispersive estimate and  Strichartz estimates for Schr\"odinger equations and the global well-posedness theory for the 3D cubic NLS.

The standard dispersive estimates and Strichartz estimates for the free Schr\"odinger operator $e^{it\Delta}$ when $d=3$ reads as follows. We refer to \cite{cazenave2003semilinear,tao2006nonlinear} for details. 

\begin{lem} [Dispersive estimate]
\label{lem:dispersive}
The linear operator $e^{it\Delta}$ in $\mathbb{R}^{3}$ satisfies the bound
\begin{equation*}
\|e^{it\Delta} f\|_{L_{x}^{\infty}} \lesssim |t|^{-\frac{3}{2}}\|f\|_{L_{x}^{1}}.
\end{equation*}
Moreover, by interpolation with the unitary relation $\|e^{it\Delta} f\|_{L_x^2} = \|f\|_{L_x^2}$, we have
\begin{equation} \label{eq:dispersive}
\|e^{it\Delta} f\|_{L_{x}^{p}} \lesssim |t|^{-d(\frac{1}{2}-\frac{1}{p})}\|f\|_{L_{x}^{p'}}
\end{equation}
for every $p  \geq 2$. 
\end{lem}

\begin{definition} \label{def:Strichartz}
	Let $q ,r \in [2, +\infty]$. We say $(q,r)$ is an admissible Strichartz pair in $\mathbb{R}^3$ if
	\begin{equation*}
	\frac{2}{q} + \frac{3}{r} = \frac{3}{2}\;.
	\end{equation*}
\end{definition}

\begin{lem}[Strichartz estimate] \label{lem:strichartz}
Let $(q,r)$ be an admissible Strichartz pair in $\mathbb{R}^3$. Then we have the bound
\begin{equation}
    \|e^{it\Delta} f\|_{L^q_t L^r_x(\mathbb{R}\times \mathbb{R}^3)} \lesssim \|f\|_{L^2(\mathbb{R}^3)}.
\end{equation}
Also, for any two Strichartz pairs $(q_1, r_1)$ and $(q_2,r_2)$, we have
\begin{equation}
    \bigg\| \int_0^t e^{i(t-s)\Delta}F(s)ds \bigg\|_{L^{q_1}_tL^{p_1}_x(\mathbb{R}\times \mathbb{R}^3)} \lesssim \|F\|_{L^{q_{2}^{'}}_{t} L^{p_{2}^{'}}_x(\mathbb{R}\times \mathbb{R}^3)}\;,
\end{equation}
where $q_2'$ and $r_2'$ are conjugates of $q_2$ and $r_2$. 
\end{lem}

Then we turn to the global theory for 3D cubic NLS \eqref{eq: cubicnls}. If one works with  initial data in $H^{1}$, the scattering result is indeed easier.
Much stronger low regularity results holds for equation \eqref{eq: cubicnls}, \cite{colliander2004global}, \cite{kenig2010scattering}, see also \cite{dodson2020global} and reference therein.

As a corollary of lower regularity results \cite{colliander2004global}, \cite{kenig2010scattering}, one has 
\begin{prop}
The initial value problem \eqref{eq: cubicnls} is globally well-posed and scatters in the  $H^1$ space. More precisely, for any $u_0$ with finite energy, $u_{0}\in H^{1}$, there exists a
unique global solution $u\in C^0_t({H}^1_x)\cap L^{5}_{x,t}$ such that
\begin{equation}
    \int_{-\infty}^{\infty} \int_{\mathbb{R}^3}|u(t,x)|^{5}dxdt\leq C(\|u_{0}\|_{H^{1}}),
\end{equation}
for some constant $C(\|u_{0}\|_{H^{1}})$ that depends only on $\|u_{0}\|_{H^{1}}$.
More precisely, $C(\|u_{0}\|_{H^{1}})$ is a polynomial of $H^{1}$-norm of the initial data.
\end{prop}
\begin{rem}
We note that the scattering norm $L^{5}_{x,t}$ can be interpolated by the interaction Morawetz bound in 3D
\begin{equation}
  \|u\|^4_{L^{4}_{t,x}} \lesssim  \|u\|^2_{L^2}\cdot \sup_{t}\|\nabla^{\frac{1}{2}}u\|^2_{L^2} \lesssim \|u\|^4_{H^1},   
\end{equation}
and the a-priori bound $\|u\|_{L^{\infty}_t H^1_x}$ according to the conservation of the energy. Thus it can be expressed as a polynomial of the $H^{1}$-norm of the initial data.

Strictly speaking, one first performs an  interpolation between  $L_{t,x}^{4}$ and $L_{t}^{\infty}L_{x}^{6}$ to obtain a control for some $L_{t}^{p}L_{x}^{q}$, so that $\frac{2}{p}+\frac{3}{q}=1$. This a-priori bound plus the local theory gives an a  priori $L_{t}^{5}L_{x}^{5}$ bound.
\end{rem}

\begin{rem}
We also note that, for NLS with criticality $s_c$ larger than the criticality of the interaction Morawetz estimate, i.e. $\frac{1}{4}$, if one assumes an a-priori bound higher than the critical level: $\|u\|_{L^{\infty}_t H^{s}_x}<\infty$ ($s>s_c$), then the 
scattering can be obtained directly via the interpolation. Thus, generally, studying the NLS model with data in a critical space is  highly nontrivial.
\end{rem}

\section{Proof of Theorem \ref{thm1}, \ref{thm2}}\label{3}
\subsection{Proof of Theorem \ref{thm1}}
We start with Theorem \ref{thm1}. This is the most non-technical part of the five theorems, but the proofs of the other theorems build upon this one.

We denote
\begin{equation}
M_{1}=\|u_{0}\|_{L^{1}_x}+(\|u_{0}\|_{H^{1}}+1)^{2}
\end{equation}
and we only consider $M_{1}$ large.
It is enough to prove 
\begin{equation}
\|u(t)\|_{L_{t}^{\infty}}\lesssim \exp M_{1}^{\alpha}	
\end{equation}
for some $\alpha>0$.

We recall that we have scattering for such data,
\begin{equation}\label{eq: scatteringnorm}
\|u\|_{L_{t}^{p}L_{x}^{q}}^{p}\lesssim M_{1}^{\beta},
\end{equation}
% the interpolation of $L_{t,x^{4}}$ and $L_{t}^{\infty}L_{x}^{6}$ gives $L_{t}^{6}L_{x}^{9/2}$
and by mass and energy conservation law, we have
\begin{equation}
\|u(t)\|_{L_{t}^{\infty}H^{1}}\lesssim M_{1}.
\end{equation}
 
We proceed with a   bootstrapping  argument for the quantity 
\begin{equation}\label{eq: ab}
A(t):=\sup_{0\leq\tau\leq t}\tau^{3/2}A(\tau).
\end{equation}

We only perform  a priori estimates for $A(t)$, i.e. we will assume $A(t)$ is finite for all $t$ and prove estimates of form, for example,
\begin{equation}
A(t)\leq C+\epsilon A(t).
\end{equation}
One can apply approximation and continuity arguments (bootstrapping argument) to transfer those a priori estimates to desired estimates, i.e. get rid of the assumption that $A(t)$ is finite for all $t$.
We will need a large parameter $M$ which will be determined later. We write down the Duhamel Formula of $u$, 
 \begin{equation}\label{eq: baseduhamel}
 u(t)=e^{it\Delta}u_{0}-i\int_{0}^{t}e^{i(t-s)\Delta}(|u(s)|^{2}u(s))ds,	
 \end{equation}
 for notation simplicity, we'll denote
 \begin{equation}
 N(u)=|u|^{2}u.	
 \end{equation}
 
 We first prove
\begin{lem}\label{lem: thm1local}
For $0<t\ll M_{1}^{-100}$,
\begin{equation}
\|u(t)\|_{L^{\infty}_x}\leq CM_{1}t^{-3/2}+Ct^{1/2}M_{1}^{3}+\frac{1}{2}A(t)t^{-3/2}.	
\end{equation} 	
\end{lem}
This lemma, though looking simple, establishes the base for the bootstrapping argument  for $A(t)$, i.e. $A(t)$ is locally finite.  But we note, this estimate is only useful for  short time.

\begin{proof}[Proof of Lemma \ref{lem: thm1local}]

By classical dispersive estimate,
\begin{equation}\label{eq: dl}
\|e^{it\Delta}u_{0}\|_{L_{x}^{\infty}}\lesssim M_{1}t^{-3/2}
\end{equation}	
and this handles the linear part in the Duhamel formula \eqref{eq: baseduhamel}.
For the nonlinear part in \eqref{eq: baseduhamel}, we may spilt the integral into 
$\int_{0}^{t}=\int_{0}
^{t/2}+\int_{t/2}^{t}$.
We estimate the first part by 
\begin{equation}\label{eq: h12control}
\int_{0}^{t/2}\|e^{i(t-s)\Delta}N(u)\|_{L^{\infty}_x}ds \lesssim t^{-3/2}\int_{0}^{t/2}\|N(u)\|_{L_{x}^{1}}ds\lesssim t^{-3/2}\frac{t}{2}\|u\|_{L_{t}^{\infty}\dot{H}^{1/2}}^{3}\lesssim t^{-1/2}M_{1}^{3}.	
\end{equation}
We use the Sobolev embedding, $W^{3+,1}\mapsto L^{\infty}$, and estimate via
\begin{equation}\label{eq: reducsobo}
\begin{aligned}
&\|\int_{t/2}^{t}e^{i(t-s)\Delta}(N(u(s)))\|_{L_{x}^{\infty}}ds\\
\lesssim &\int_{t/2}^{t}\|e^{it\Delta}\langle \nabla \rangle N(u)\|_{L_{x}^{3+}}ds\\
\lesssim &\int_{t/2}^{t}(t-s)^{-(\frac{1}{2}+)}\|\langle \nabla \rangle N(u)\|_{L_{x}^{\frac{3}{2}-}}ds
\end{aligned}	
\end{equation}
and we apply again dispersive estimate in the last step.

Applying fractional Leibniz rule\footnote{It is somehow important here that the Leibniz rule covers the end point involves $L^{\infty}$.} when $s<1$, Theorem A-12 in \cite{kenig1993well}, or the usual Leibniz rule when $s=1$, we have
\begin{equation}\label{eq: fl}
\|\langle \nabla \rangle N(u)\|_{L_{x}^{p}}=\|u\|_{H^{s}}\|u\|_{L_x^{\infty}}\|u\|_{L_{x}^{q}}, \quad \forall \frac{1}{p}=\frac{1}{q}+\frac{1}{2}, \quad 1<p,q<\infty.
\end{equation}

One may now carry out the estimate \eqref{eq: reducsobo} as 
\begin{equation}\label{eq: afterredu}
\begin{aligned}
&\int_{t/2}^{t}(t-s)^{-(\frac{1}{2}+)}\|\langle \nabla \rangle N(u)\|_{L_{x}^{\frac{3}{2}-}}ds\\
&\lesssim (\int_{t/2}^{t}(t-s)^{-1+})^{\frac{1}{2-}}(\int_{0}^{t/2}\|\langle \nabla \rangle N(u)\|_{L_{x}^{\frac{3}{2}-}}^{2+})^{\frac{1}{2+}}ds\\
&\lesssim t^{\epsilon}M_{1}^{2}(\int_{t/2}^{t}\|u(s)\|_{L_x^{\infty}}^{2+}\|u\|_{L_{x}^{6-}}^{2+}ds)^{\frac{1}{2+}}.
\end{aligned}	
\end{equation}

Summarizing \eqref{eq: reducsobo}, \eqref{eq: afterredu}, and letting  be 
$(4+,6-)$ be $\dot{H}^{1/2}$ admissible pair, we have
\begin{equation}\label{eq: finalredu}
\begin{aligned}
\big\|\int_{t/2}^{t}e^{i(t-s)\Delta}N(u(s))ds\big\|_{L_{x}^{\infty}}&\lesssim    t^{\epsilon}M_{1}^{2}(\int_{t/2}^{t}\|u(s)\|_{L_{x}^{\infty}}^{2+}\|u\|_{L_{x}^{6-}}^{2+}ds)^{\frac{1}{2+}}\\
&\lesssim t^{\epsilon}M_{1}^{2}(t)^{\frac{1}{2}-}(\int^{t}_{t/2}\|u\|_{L_{x}^{\infty}}^{4+}\|u\|_{L_{x}^{6-}}^{4+}ds)^{\frac{1}{4+}}.
\end{aligned}	
\end{equation}

Thus, we have, by \eqref{eq: dl}, \eqref{eq: h12control}, \eqref{eq: finalredu},  for $t\ll M_{1}^{-100}$,
\begin{equation}
\|u(t)\|_{L_{x}^{\infty}}\leq CM_{1}t^{-3/2}+Ct^{-1/2}M_{1}^{2}+CM_{1}^{2}t^{1/2-}A(t)(\int_{t/2}^{t}\|u\|_{L_{x}^{6-}}^{4+}ds)^{\frac{1}{4+}}.
\end{equation}

Lemma \ref{lem: thm1local} follows, since $t\leq M_{1}^{-100}$ and $M_{1}$ large.
\end{proof}

Next step, we want to handle the case $t\leq 2M$, and we need to refine the proof of Lemma \ref{lem: thm1local}. We state the next lemma as follows.
\begin{lem}\label{lem: thm1local2}
For $t\leq 2M$,
\begin{equation}
\|u(t)\|_{L_{x}^{\infty}}\leq CM_{1}t^{-3/2}+M_{1}^{53}+\frac{1}{2}A(t)t^{-3/2}.	
\end{equation}
\end{lem}
\begin{proof}
We observe that conservation laws give,
\begin{equation}
\|u(t)\|_{L_{t}^{\infty}H^{1}} <\infty.
\end{equation}

Thus, for any $t$, one has 
\begin{equation}
\|u(s)\|_{L_{t}^{4+}L_{x}^{6-}}\lesssim M_{1}^{-10}.
\end{equation}

We still use Duhamel formula, \eqref{eq: baseduhamel}, and estimate the linear part  as  in \eqref{eq: dl}. Then  we split the nonlinear part into $\int_{0}^{t}=\int_{0}^{t-M_{1}^{-100}}$ and $\int_{t-M_{1}^{-100}}^{t}$.
Now, for the first part we proceed as in  \eqref{eq: h12control} via
\begin{equation}
\begin{aligned}
&\big\|\int_{0}^{t-M_{1}^{-100}}e^{i(t-s)\Delta}N(u(s))ds\big\|_{L_{x}^{\infty}}\\
\lesssim &\int_{0}^{t-M_{1}^{-100}}(t-s)^{3/2}\|u(s)\|^{3}_{L_{x}^{3}}ds\\
\lesssim &M_{1}^{50}M_{1}^{3},
\end{aligned}
\end{equation}
and the for the second part as in  \eqref{eq: reducsobo}, \eqref{eq: afterredu}, and we derive
\begin{equation}\label{eq: secrefine}
\begin{aligned}	
\big\| \int_{t-M_{1}^{-100}}^{t}e^{i(t-s)\Delta}N(u(s))ds \big\|_{L_{x}^{\infty}}\lesssim (M_{1}^{-100})^{\frac{1}{2}-}M_{1}^{2}(\int_{t-M_{1}^{-100}}^{t}\|u(s)\|_{L_{x}^{\infty}}^{4+}\|u\|_{L_{x}^{6-}}^{4+}ds)^{\frac{1}{4+}}.
\end{aligned}
\end{equation}
Thus, the proof of Lemma \ref{lem: thm1local2} is now complete.
\end{proof}

Now we turn to another lemma for the case $t\geq 2M$. The statement reads,
\begin{lem}\label{lem: eapriori}
For $t\geq 2M$,
\begin{equation}\label{eq: eseap}
\|u(t)\|_{L_{x}^{\infty}}\leq CM_{1}t^{-3/2}+CM_{1}^{7}t^{-3/2}+\frac{1}{2}A(t)t^{-\frac{3}{2}}+CM_{1}^{4+}t^{-3/2}(\int_{t-M}^{t}A(s)^{4+}\|u(t)\|_{L_{x}^{6-}}^{4+}ds)^{\frac{1}{4+}}.
\end{equation}	
\end{lem}

\begin{proof}
To deal with this case, we need to, similar as in \cite{fan2021decay},  further write \eqref{eq: baseduhamel} into 
\begin{eqnarray*}
u(t)&=&e^{it\Delta}u_{0}-i\int_{0}^{M}e^{i(t-s)\Delta}N(u)ds-i\int_{M}^{t-M}e^{i(t-s)\Delta}N(u(s))ds-i\int_{t-M}^{t}e^{i(t-s)\Delta}N(u(s))ds\\
&:=&e^{it\Delta}u_{0}+F_{1}+F_{2}+F_{3}.
\end{eqnarray*}

The estimate of $F_{1}$ will be straightforward, since $t-M\gtrsim t$, and we have, via simple Minkowski's inequality and the usual dispersive estimate,
\begin{equation}\label{eq: estf1}
\|F_{1}\|_{L_{x}^{\infty}}\lesssim t^{-3/2}\int_{0}^{M}\|u(s)\|_{L_{x}^{3}}^{3}ds \lesssim t^{-3/2}MM_{1}^{3}.
\end{equation}

For $F_{2}$, it is crucial  that we are in dimension\footnote{One can see this scheme will have a log divergence if instead one considers cubic NLS in dimension 2.} at least 3, we have
\begin{equation}
\begin{aligned}
\|F_{2}\|_{L_{x}^{\infty}}&\leq \int_{M}^{t-M}\|e^{i(t-s)\Delta}N(u(s))\|_{L_{x}^{\infty}}ds\\
&\lesssim \int_{M}^{t-M}(t-s)^{-3/2}\|u\|_{L_{x}^{\infty}}\|u\|_{L_{x}^{2}}^{2}ds\\
&\lesssim \int_{M}^{t-M}(t-s)^{-3/2}A(s)s^{-\frac{3}{2}}M_{1}^{2 }ds.
\end{aligned}	
\end{equation}
Thus, we have
\begin{equation}\label{eq: f2absorb}
\|F_{2}\|_{L_{x}^{\infty}}\leq \frac{1}{2}A(t)t^{-3/2}, 	
\end{equation}
for all $M\geq CM_{1}^{4}$ for some universal $C$ large.

For $F_{3}$, the main point here is that we use analysis similar to Lemma \ref{lem: thm1local}, \ref{lem: thm1local2} to refine the analysis in \cite{fan2021decay}, and we will further control via the Gronwall's argument.

Now, similar to the proof of Lemma \ref{lem: thm1local2}, in particular the part for \eqref{eq: secrefine}, we derive
\begin{equation}\label{eq: controlforf3}
\begin{aligned}
\|F_{3}\|_{L_{x}^{\infty}}\lesssim \|F_{3}\|_{W^{3+,1}}\lesssim M^{\frac{1}{2}+}M_{1}^{2}(\int_{t-M}^{t}\|u(s)\|_{L_{x}^{\infty}}^{4+}\|u\|_{L_{x}^{6-}}^{4+}ds)^{\frac{1}{4+}}.
\end{aligned}
\end{equation}

To summarize, using \eqref{eq: estf1}, \ $M\sim M_{1}^{4}$ so that \eqref{eq: f2absorb} holds,  and \eqref{eq: controlforf3}, we obtain Lemma \ref{lem: thm1local2}.
\end{proof}

Based on the above three lemmas, we are now ready to prove Theorem \ref{thm1}.

\begin{proof}[Concluding the proof of Theorem \ref{thm1}]
By Lemma \ref{lem: thm1local2}, one has, for some $\beta_{1}>0$,
\begin{equation}\label{eq: atshort}
A(t)\leq CM_{1}^{\beta_{1}}, \quad \forall t\leq CM_{1}^{4},	
\end{equation}
and by Lemma \ref{lem: eapriori}, one has, for some $\beta_{2}, \beta_{3}>0$,
\begin{equation}\label{eq: atlong}
\begin{aligned}
A(t)&\leq CM_{1}^{\beta_{2}}+CM_{1}^{\beta_{3}}(\int_{t-M}^{t}A(s)^{4+}\|u(t)\|_{L_{x}^{6-}}^{4+}ds)^{\frac{1}{4+}}\\
&\leq CM_{1}^{\beta_{2}}+CM_{1}^{\beta_{3}}(\int_{0}^{t}A(s)^{4+}\|u(t)\|_{L_{x}^{6-}}^{4+}ds)^{\frac{1}{4+}}.\end{aligned}	
\end{equation}
Thus, for all $t$,
\begin{equation}\label{eq: firsta}
A(t)^{4+}\leq CM_{1}^{\beta_{4}}+CM_{1}^{\beta_{4}}\int_{0}^{t}A(s)^{4+}\|u(s)\|_{L_{x}^{6-}}^{4+}ds.	
\end{equation}
Via Gronwall's inequality and \eqref{eq: scatteringnorm}, one has 
\begin{equation}
	A(t)\leq C\exp M_{1}^{\beta}.
\end{equation}
This gives Theorem \ref{thm1}.
\end{proof}
\begin{proof}[An explanation for Corollary \ref{cor: thm1}].
We now briefly explain the point of Corollary \ref{cor: thm1} as follows. We still define $A(t)$ in the same way. In view of the scattering result (scattering norm is finite), letting $t$ big enough, the quantity $(\int_{t-M}^{t}\|u(t)\|_{L_{x}^{6-}}^{4+}ds)^{\frac{1}{4+}}$ can be made arbitrarily small, say $\epsilon>0$.\vspace{3mm}

Thus, consider the first inequality in \eqref{eq: atlong}, letting $t$ big enough, we can drag $A(s)^{4+}$ out and letting $\epsilon$ small enough to beat other constants,
\begin{equation}
\begin{aligned}
A(t)&\leq CM_{1}^{\beta_{2}}+CM_{1}^{\beta_{3}}(\int_{t-M}^{t}A(s)^{4+}\|u(t)\|_{L_{x}^{6-}}^{4+}ds)^{\frac{1}{4+}}\\
&\leq CM_{1}^{\beta_{2}}+CM_{1}^{\beta_{3}}A(t)(\int_{0}^{t}\|u(t)\|_{L_{x}^{6-}}^{4+}ds)^{\frac{1}{4+}}\\
&\leq CM_{1}^{\beta_{2}}+CM_{1}^{\beta_{3}}A(t)\epsilon \\
&\leq CM_{1}^{\beta_{2}}+\frac{1}{2} A(t).
\end{aligned}	
\end{equation}
Then moving the second term on the right hand side to the left, this gives us a polynomial bound for $\limsup_{t\rightarrow \infty} A(t)$ as desired.
In many situations one may be particularly  interested in the asymptotic behavior of solutions for $t$ large rather than estimates that are uniform with respect to  $t$. In our case we are for example interested on  $\sup_{t>0}t^{3/2}\|u(t)\|_{L_{x}^{\infty}}$ versus $\limsup_{t\rightarrow \infty}t^{3/2}\|u(t)\|_{L_{x}^{\infty}}$. Thus the above observation may be useful for some problems. One may also understand this decay result in the following manner: the exponential constant dependence is caused by the finite time, for the long time, the constant dependence is essentially polynomial.

The proof of Corollary \ref{cor: thm1} is now complete.
\end{proof}
\subsection{Proof of Theorem \ref{thm2}}
One can obtain a polynomial type control rather than an exponential one, with the extra assumption $xu_{0}\in L_{x}^{2}$. This is because in this case one can apply  the pseudo-conformal transformation to get a quantitative control for the decay of $L^{6}$. The argument below seems classical, see for example, \cite{murphy2017subcritical}.
Indeed, let $J(t)=x+2it\nabla$, one has 
\begin{equation}\label{eq: func}
\|f\|_{L_{x}^{6}}\lesssim t^{-1}\|J(t)f\|_{L_{x}^{2}}.
\end{equation}
Furthermore, for $u$ solving \eqref{eq: cubicnls}, the quantity
\begin{equation}\label{eq: pesudoenergy1}
	\frac{1}{2}\|J(t)f\|^{2}_{L_{x}^{2}}+\int t^{2}|u|^{4}dx
\end{equation}
is  monotonically decreasing in $t$, and when $t=0$, it equals $\|xu\|_{L^2_x}^{2}$.
Thus, let 
\begin{equation}
M_{1}:=(1+\|u_{0}\|_{H_{1}})^{2}+\|xu_{0}\|_{L^2_x}+\|u_{0}\|_{L_{x}^{1}},	
\end{equation}
one has 
\begin{equation}
\|u(t)\|_{L_{x}^{6}}\lesssim t^{-1}M_{1},	
\end{equation}
and a simple interpolation with mass conservation gives
\begin{equation}
\|u(t)\|_{L_{x}^{6-}}\lesssim t^{-(1-)}M_{1}.	
\end{equation}
Now we go back to \eqref{eq: atshort}, \eqref{eq: atlong}, and we enhance \eqref{eq: atlong} into
\begin{equation}\label{eq: atlongrefined}
\begin{aligned}
A(t)\leq 	CM_{1}^{\beta_{2}}+CM_{1}^{\beta_{3}}(\int_{t-M}^{t}t^{-(4-)}A(s)^{4^{+}}ds)^{\frac{1}{4+}}\\
\leq CM_{1}^{\beta_{2}}+CM_{1}^{\beta_{3}}(\int_{0}^{t}(M+s)^{-(4-)}A(s)^{4^{+}}ds)^{\frac{1}{4+}}.
\end{aligned}
\end{equation}
Here we use the fact that for $s\in [t-M,t], \, s \sim t$, since $t\geq 2M.$
Thus, rather than \eqref{eq: firsta}, we have the estimate for $A$
\begin{equation}\label{eq: seconda}
A(t)^{4+}\leq C(M_{1}+M)^{\beta_{4}}+CM_{1}^{\beta_{4}}\int_{0}^{t}(M+t)^{-(4-)}A(s)^{4+}ds.
\end{equation}
Here we need choose $M$  to be a large polynomial of $M_{1}$, so that 
\begin{equation}
M_{1}^{\beta^{4}}\int_{0}^{\infty}(s+M)^{-4(-)}ds\lesssim 1,
\end{equation}
and Gronwall's inequality for \eqref{eq: seconda} gives
\begin{equation}
A(t)\lesssim M_{1}^{\beta}.	
\end{equation}

\section{Proof of Theorem \ref{thm3}}\label{4}
We now turn to the proof of Theorem \ref{thm3}. We first note that in general it is very hard to study unconditional scattering for \eqref{eq: cubicnls} for initial data $u_{0}\in H^{s}, s>1/2$, and it is a major open problem to study global well-posedness for $u_{0}\in H^{1/2}$. Both are not hard if one further assumes  that the initial data is in $xu_{0}\in L^{2}$.

Assume  $u$ solves \eqref{eq: cubicnls} with initial data
\begin{equation}
\|u_{0}\|_{H^{1}}+\|xu_{0}\|_{L_{x}^{2}}\leq M_{1}<\infty.
\end{equation}
We have, for some $\beta_{0}>0$
\begin{equation}\label{eq: scatterandregular}
\begin{aligned}
&\|u\|_{L_{t,x}^{5}}^{5}\leq CM_{1}^{\beta_{0}},\\
&\|u\|_{L_{t}^{\infty}H^{s}}\leq Ce^{M_{1}^{\beta_{0}}}.
\end{aligned}
\end{equation}
The proof of the first estimate  in \eqref{eq: scatterandregular} may be classical from the pseudo-conformal symmetry, see  for example textbook \cite{tao2006nonlinear}, \cite{bourgain1998refinements}, we briefly sketch it below for the convenience of the readers.
The second estimate in \eqref{eq: scatterandregular} follows from the first one by classical persistence of regularity arguments, see for example textbook \cite{tao2006nonlinear},   Lemma 3.12 in \cite{colliander2008global}, (by slightly changing the value of $\beta_{0}$ if necessary.)

Fix $s>1/2$, by the local theory, we have that there is a $\delta>0$, with $\delta\sim M_{1}^{-\beta_{1}}$, so that $u$ is well-posed, and 
\begin{equation}
\|u\|_{L_{t,x}^{5}[0,\delta]}\leq 1.	
\end{equation}
Meanwhile, by pseudo-conformal transform, we may define
\begin{equation}
\tilde{u}(s)=\frac{1}{(-s)^{3/2}}u(-\frac{1}{s},\frac{x}{-s})e^{-i|x|^{2}s/4},\end{equation}
and note that if $u$ solves \eqref{eq: cubicnls} in $[0,\delta]$ then $\tilde{u}$ is well defined in $(-\infty, -1/\delta]$, and solves
\begin{equation}
i\partial_{s}\tilde{u}+\Delta \tilde{u}=-s|\tilde{u}|^{2}\tilde{u}.	
\end{equation}
Consider
\begin{equation}
\begin{aligned}
H(\tilde{u}(s)):=\frac{1}{2}\int |\nabla \tilde{u}(s)|^{2}dx+\int \frac{1}{4}(-s)|\tilde{u}|^{4} dx\\
\equiv \frac{1}{8}\|J(t)u(t)\|_{L^2_x}^{2}+\frac{1}{4}\int t^{2}|u|^{4}dx
\end{aligned}
\end{equation}
where one does change of variable $t=-\frac{1}{s}$, and recall we have $J(t)=x+2it\nabla$.
Then we see that $H$ is monotonically decreasing, and 
\begin{equation}
\lim_{s\rightarrow -\infty}H(\tilde{u}(s))=\lim_{t\rightarrow 0}	\frac{1}{8}\|J(t)u(t)\|_{L^2_x}^{2}+\frac{1}{4}\int t^{2}|u|^{4}dx=\|xu_{0}\|_{L_{x}^{2}}^{2}.
\end{equation}
Thus, we have
\begin{equation}
H(\tilde{u}(-\frac{1}{\delta}))\leq CM_{1}^{2},	
\end{equation}
and hence  $\tilde{u}$ is a global solution with $\|u\|_{L_{t}^{\infty}H^{1}}\leq CM_{1}$.
Moreover, the standard $H^{1}$ local theory for  the usual cubic NLS gives
\begin{equation}
\|\tilde{u}\|_{L_{t,x}^{5}[-\frac{1}{\delta},0]}\leq CM_{1}^{\beta_{2}},
\end{equation}
since $(-s)$ is bounded by $-\frac{1}{\delta}\sim M_{1}^{\beta_{1}}$ for $s\in [-\frac{1}{\delta},0]$.
Thus, one further obtains 
\begin{equation}
\|u\|_{L_{t,x}^{5}[\frac{1}{\delta}, 0)}\leq CM_{1}^{\beta_{3}}.	
\end{equation}
This gives \eqref{eq: scatterandregular}.

Furthermore, we still derive $\|J(t)u\|_{L^2_x}\leq CM_{1}$, thus, by \eqref{eq: func}, we obtain,
\begin{equation}\label{eq: l6decay}
\|u(t)\|_{L_{x}^{6}}\leq CM_{1}t^{-1}.	
\end{equation}

We can now plug in \eqref{eq: scatterandregular}, \eqref{eq: l6decay} in the scheme of the proof of Theorem \ref{thm2}. We will sketch the argument  highlighting  the modifications that  need to be made.
We still focus on a priori estimates for $A(t)$, \eqref{eq: ab}, and we pose
\begin{equation}
(\|u_{0}\|_{H^{1}}+1)^{2}+\|u_{0}\|_{L_x^{1}}+\|xu_{0}\|_{L_x^{2}}=M_{1}.
\end{equation}
We start with an analogue of Lemma \ref{lem: thm1local},
 \begin{lem}\label{lem: thm3local}
 There exists $\beta_{4}>0$, $\beta_{4}$ large, so that for $t\leq M_{1}^{-\beta_{4}}$, one has 
 \begin{equation}\label{eq: t3l1}
 \|u(t)\|_{L_{x}^{\infty}}\leq CM_{1}t^{-3/2}+CM_{1}^{3}t^{-\frac{1}{2}}.	
 \end{equation}
 \end{lem}
\begin{proof}[Proof of Lemma \ref{lem: thm3local}]
Recall Duhamel formula \eqref{eq: baseduhamel}, the linear part is still controlled via \eqref{eq: dl}.
By the local theory of \eqref{eq: cubicnls}, we have
\begin{equation}\label{eq: lt}
\|u\|_{L^{\infty}_{t}H^{s}[0,M_{1}^{-100}]}\lesssim M_{1}.	
\end{equation}
For the nonlinear part, we still split $\int_{0}^{t}=\int_{0}^{t/2}+\int_{t/2}^{0}$.
The first part is still controlled via
\begin{equation}\label{eq: ta}
\big\| \int_{0}^{t/2}e^{i(t-s)\Delta}N(u)ds \big\|_{L^{\infty}_x} \lesssim t^{-3/2}\int_{0}^{t/2}\|u(s)\|_{L_{x}^{3}}^{3}ds\lesssim t^{-1/2}M_{1}^{3}.
\end{equation}
In the last step of \eqref{eq: ta} we applied \eqref{eq: lt}.
For the second term, we are in some sense at the end point case when $s$ approaches $\frac{1}{2}$. Before we present more details, we want to mention that  one will see, since we are fixing  $s>\frac{1}{2}$, that we just stay away from the end point and we always have some rooms\footnote{ The easiest way to do a first check of the computation is to neglect $log$ convergence and pose all the $\kappa_{i}$ below as zero.}. 

We now analyze the second term. We pick $0<\kappa_{1},\kappa_{2}\ll 1$ small and fix them. We will decide their relative size later.

We may assume $\kappa_{1}\leq \frac{1}{2}(s-\frac{1}{2})$.

We use Sobolev embedding,
\begin{equation}
W^{\frac{1}{2}+\kappa_{1}, p_{1}}\rightarrow L^{\infty}	
\end{equation}
where $p_{1}=6-\kappa_{2}$, satisfying
\begin{equation}\label{eq: k1}
\frac{3}{p_{1}}< \frac{1}{2}+\kappa_{1}.
\end{equation}
(It is enough to further assume $\kappa_{1}>3\kappa_{2}$ for \eqref{eq: k1} to hold.)
We now estimate
\begin{equation}\label{eq: cs1}
\begin{aligned}
&\big\|\int_{\frac{t}{2}}^{t}e^{i(t-s)}N(u(s))ds\big\|_{L_{x}^{\infty}}\\
\leq &C\int_{t/2}^{t}\|e^{it\Delta}N(u)ds\|_{W^{\frac{1}{2}+\kappa_{1},p_{1}}}ds\\
\leq  &C\int_{t/2}^{t} (t-s)^{1+\frac{-2\kappa_{2}}{2(6-\kappa_{2})}}\|N(u)\|_{W^{s, p_{1}'}}ds.
\end{aligned}
\end{equation}
In the last step we apply dispersive estimate, and recall
$\|e^{it\Delta}\|_{L^{p_{1}'}\rightarrow L^{p_{1}}}\lesssim t^{-(1-\frac{\kappa_{2}}{2(6-\kappa_{2})})}$.
Now, we plug in the estimate
\begin{equation}
\|N(f)\|_{W^{s, p_{1}'}}\leq C\|f\|_{H^{s}}\|f\|_{L_{x}^{\infty}}\|f\|_{L_{x}^{p_{2}}},
\end{equation}
where $\frac{1}{p_{2}}+\frac{1}{2}=\frac{1}{p_{1}'}$, and one computes as  $p_{2}=\frac{4-\kappa_{2}}{12-2\kappa_{2}}=3-$, where we  have used the fractional Leibniz rule, Theorem A12 in \cite{kenig1993well}.
We now continue the estimate \eqref{eq: cs1} as 
\begin{equation}\label{eq: cs2}
\begin{aligned}
\leq & C\int_{t/2}^{t}(t-s)^{1-\frac{\kappa_{2}}{2(6-\kappa_{2})}}\|u\|_{H^{s}}\|u\|_{L_x^{\infty}}\|u\|_{L_x^{p_{2}}}ds\\
\leq &CA(t)t^{-3/2}\int_{t/2}^{t}(t-s)^{1-\frac{\kappa_{2}}{2(6-\kappa_{2})}}\|u(t)\|_{H^{s}}^{2} ds
\leq &CA(t)t^{-3/2}t^{\frac{\kappa_{2}}{2(6-\kappa_{2})}}M_{1}^{2},
\end{aligned}	
\end{equation}
since $t\leq M_{1}^{-\beta_{4}}$, thus when $\beta_{4}$ is large enough, one has \eqref{eq: cs1} bounded by $\frac{1}{2}A(t)t^{-3/2}$.\\
Combining this with \eqref{eq: ta}, \eqref{eq: dl}, Lemma \ref{lem: thm3local} follows.
\end{proof}

 We now cover the part when $t\leq 2M$, and again $M$ is a large number which will be chosen later.
\begin{lem}\label{lem: thm3local2}
For $t\leq 2M$, one has, for some $\beta_{5}$ large, 
\begin{equation}\label{eq: t3l2}
\|u(t)\|_{L_{t}^{\infty}}\leq M_{1}t^{-\frac{3}{2}}+(M_{1}^{\beta_{5}}+\ln M)e^{3M_{1}\beta_{0}}+\frac{1}{2}A(t)t^{-\frac{3}{2}}.
\end{equation}
\end{lem}
\begin{proof}[Proof of Lemma \ref{lem: thm3local2}]
We may assume $t\geq M_{1}^{-\beta_{4}}$, otherwise we use Lemma \ref{lem: thm3local}.  Recall again \eqref{eq: baseduhamel}, this time, we split the nonlinear part as $\int_{0}^{t}=\int_{0}^{t-e^{-M_{1}^{\beta_{6}}}} $ and $\int_{t-e^{-M_{1}^{\beta_{6}}}}^{t}$.

Now, one estimates the first part as 
\begin{equation}\label{eq: againfirst}
\int_{0}^{t-e^{-M_{1}^{\beta_{6}}}}\|e^{i(t-s)}N(u)\|_{L_{x}^{\infty}}ds\\
\leq C\int_{0}^{t-e^{-M_{1}^{\beta_{6}}}}(t-s)^{-1}\|u\|_{H^{s}}^{3}ds\\
\leq C(M_{1}^{\beta_{6}}+\ln M)e^{3M_{1}^{\beta_{0}}}.
\end{equation}
In the last step we have plugged in \eqref{eq: scatterandregular}.
For the second part, we estimate similarly as \eqref{eq: cs1}, \eqref{eq: cs2}, (choosing $\kappa_{1}, \kappa_{2}$ as in the proof of Lemma \ref{lem: thm3local}), 
\begin{equation}
\begin{aligned}
&\|\int_{t-e^{-M_{1}^{\beta_{6}}}}^{t}e^{i(t-s)\Delta}N(u(s))ds\|_{L_{x}^{\infty}}\\
\leq &C\int_{t-e^{-M_{1}^{\beta_{6}}}}^{t}(t-s)^{1-\frac{\kappa_{2}}{2(6-\kappa_{2})}}\|u\|_{H^{s}}^{2}ds\\
\leq &Ce^{-\frac{\kappa_{2}}{2(6-\kappa_{2})}M_{1}^{\beta_{6}}}e^{2M_{1}^{\beta_{0}}}A(t)t^{-3/2}.
\end{aligned}
\end{equation}
Thus, when $\beta_{6}$ is large enough, so that $Ce^{-\frac{\kappa_{2}}{2(6-\kappa_{2})}M_{1}^{\beta_{6}}}e^{2M_{1}^{\beta_{0}}}\leq \frac{1}{2}$, then one can combine it with \eqref{eq: againfirst} and \eqref{eq: dl} to obtain the desired estimate.

\end{proof}

We now present an analogue of Lemma  \ref{lem: eapriori},
\begin{lem}\label{lem: thm3eap}
For $t\geq 2M$, one has 
\begin{equation}\label{eq: t3l3}
\|u(t)\|_{L_{x}^{\infty}}\leq Ct^{-3/2}Me^{3M_{1}^{\beta_{0}}}+\frac{1}{2}A(t)t^{-3/2}+ CA(t)t^{-\frac{3}{2}}e^{M_{1}^{\beta_{0}}}M^{-(\frac{1}{2}-)}.
\end{equation}
\end{lem}
\begin{proof}[Proof of Lemma \ref{lem: thm3eap}]
We write again,
\begin{eqnarray*}
u(t)&=&e^{it\Delta}u_{0}-i\int_{0}^{M}e^{i(t-s)\Delta}N(u)ds-i\int_{M}^{t-M}e^{i(t-s)\Delta}N(u(s))ds-i\int_{t-M}^{t}e^{i(t-s)\Delta}N(u(s))ds\\
&:=&e^{it\Delta}u_{0}+F_{1}+F_{2}+F_{3},
\end{eqnarray*}
and estimate $F_{1},F_{2},F_{3}$.
The terms $F_{1}, F_{2}$ will be estimated similarly as in the proof of Lemma \ref{lem: eapriori}.
We have for $F_{1}$,
\begin{equation}
\|F_{1}(t)\|_{L_{x}^{\infty}}ds\leq \int_{0}^{M}\|e^{i(t-s)\Delta}N(u(s))\|_{L_{x}^{\infty}}\leq Ct^{-3/2}M\|u\|_{L_{t}^{\infty}H^{s}}^{3}\leq Ct^{-3/2}Me^{3M_{1}^{\beta_{0}}}.	
\end{equation}
We have for $F_{2}$, exactly as in the proof of \eqref{eq: f2absorb}
\begin{equation}\label{eq: f2absorbonemore}
\|F_{2}\|_{L_{x}^{\infty}}\leq \frac{1}{2}A(t)t^{-3/2}.
\end{equation}
For $F_{3}$, compared to the proof of Lemma \ref{lem: thm3local}, \ref{lem: thm3local2}, it is important now for us to also apply \eqref{eq: l6decay} in the estimate of $F_{3}$.
Recall $0<\kappa_{1}, \kappa_{2}\ll 1$ defined in the proof of Lemma \ref{lem: thm3local}, and $p_{1}=6-\kappa_{2}, p_{2}=\frac{4-\kappa_{2}}{12-2\kappa_{2}}$,
we estimate $F_{3}$ via
\begin{equation}\label{eq: f3onemore}
\begin{aligned}
&\int_{t-M}^{t}e^{i(t-s)\Delta}N(u(s))\|_{L_{x}^{\infty}}ds\\
\leq &C\int_{t-M}^{t}\|e^{i(t-s)\Delta}N(u(s))\|_{W^{\frac{1}{2}+\kappa_{1},3}}ds\\
\leq &C\int_{t-M}^{t}(t-s)^{-(1-\frac{\kappa_{2}}{2(6-\kappa_{2})})}\|u\|_{H^{s}}\|u\|_{L_x^{p_{2}}}\|u\|_{L_x^{\infty}}ds\\
\leq &CA(t)t^{-3/2}t^{-(\frac{1}{2}-\kappa_{3}))}M_{1}e^{M_{1}^{\beta_{0}}}\int_{t-M}^{t}(t-s)^{-(1-\frac{\kappa_{2}}{2(6-\kappa_{2})})}ds\\
\leq &CA(t)t^{-\frac{3}{2}}t^{-(\frac{1}{2}-\kappa_{2})}M^{\frac{\kappa_{2}}{2(6-\kappa_{2})}}.
\end{aligned}
\end{equation}
Above we used a  simple interpolation to conclude that 
$$\|u\|_{L_x^{p_{2}}}=\|u\|_{L_x^{3-}}\leq C\|u\|_{L_x^2}^{\frac{1}{2}+}\|u\|_{L_x^6}^{\frac{1}{2-}}\leq t^{-(\frac{1}{2}-\kappa_{3})}M_{1}$$ for some $\kappa_{1}>0$ and we plug the result  in \eqref{eq: scatterandregular}.

Now, given $t\geq 2M$, one summarizes estimate \eqref{eq: f3onemore} as 
\begin{equation}\label{eq: f3}
\|F_{3}\|_{L_{x}^{\infty}}\leq CA(t)t^{-\frac{3}{2}}e^{M_{1}^{\beta_{0}}}M^{-(\frac{1}{2}-)}.	
\end{equation}
\end{proof}
Summarizing Lemma \ref{lem: thm3local}, \ref{lem: thm3local2}, \ref{lem: thm3eap}, 
 choosing $M=e^{M_{1}^{\beta_{8}}}$ so that \eqref{eq: f2absorbonemore}  holds, we have that \eqref{eq: f3onemore} reads
 as 
 \begin{equation}
 \|F_{3}(t)\|_{L_{x}^{\infty}}\leq \frac{1}{10}A(t)t^{-3/2}.	
 \end{equation}
 As a consequence we have
 \begin{equation}
 u(t)\leq CM_{1}t^{-3/2}+M^{\beta_{9}}e^{CM_{1}^{\beta_{0}}}t^{-\frac{3}{2}}+\frac{3}{2}A(t)t^{-3/2}
 \end{equation}
 i.e.
 \begin{equation}
 A(t)\leq CM^{\beta_{9}}e^{CM_{1}^{\beta_{0}}}\lesssim e^{M_{1}^{\beta}},
 \end{equation}
for some $\beta>0$. This concludes the proof of Theorem \ref{thm3}.

\section{Proof of Theorem \ref{thm4}}\label{5}
What we want to present here, is that one can systematically remove the $L^{1}$ assumptions in the Theorem \ref{thm1}, \ref{thm2}, \ref{thm3}, by randomizing the initial data.
We will only prove the case for estimate \eqref{eq: thm1ran}, the other two cases can be generalized from Theorem \ref{thm2}, \ref{thm3} respectively.

We will fix a constant $A$, large.
We will use the terminology $A$-certain in \cite{de2002effect}, i.e we say an event is $A$-certain if it holds up to a set with small probability $e^{-A^{\alpha}}$. (The exact value of $\alpha$ may change line by line, but at the end, one only choose the smallest $\alpha$ involved.)

While we cannot conclude\footnote{We cannot even conclude \eqref{eq: rdispersive} if one replaces $\frac{3}{2}-$ for  $\frac{3}{2}$).} that \eqref{eq: rdispersive} holds  $A$-certainly,\begin{equation}\label{eq: rdispersive}
\|e^{it\Delta}u_{0}^{\omega}\|_{L_{x}^{\infty}}\lesssim A^{\beta}t^{-(\frac{3}{2})}, \forall t\geq 0.	
\end{equation}

We can prove, in some sense, a time average version of \eqref{eq: rdispersive}. We present more details below.
  We need to introduce a weight
 \begin{equation}\label{eq: weight}
 \gamma_{p,\epsilon}=\frac{t^{100}}{1+t^{100}}t^{3(\frac{1}{2}-\frac{1}{p})-\epsilon}.	
 \end{equation}
 
 \begin{lem}\label{lem: ardis}
 For all $2<p<\infty$, for all $\epsilon_{1}>\epsilon_{2}$, then $A$-certainly\footnote{The $\alpha$ depend on $p,\epsilon_{1}, \epsilon_{2}$ though.}, 
 \begin{equation}\label{eq: ardis}
 \|\gamma_{p,\epsilon_{1}}e^{it\Delta}u_{0}^{\omega}\|_{L_{t}^{\epsilon_{2}^{-1}}L_{x}^{p}[0,\infty)}\leq A,
 \end{equation}
 similarly, one also has 
 \begin{equation}\label{eq: ardis2}
 \|\gamma_{p,\epsilon_{1}}e^{it\Delta}\nabla u_{0}^{\omega}\|_{L_{t}^{\epsilon_{2}^{-1}}L_{x}^{p}[0,\infty)}\leq A.
 \end{equation}
 \end{lem}

\begin{rem}\label{remweight}
Regarding the weight in \eqref{eq: weight}, one notes that if for some $f(t)\lesssim t^{-3(\frac{1}{2}-\frac{1}{p})}$, then
\begin{equation}
\|\gamma_{p,\epsilon}(t)f(t)\|_{L_{t}^{\frac{1}{\epsilon'}}}\lesssim 1, \quad\forall \epsilon'<\epsilon.
\end{equation}
Meanwhile, if $\|\gamma_{p,\epsilon}(t)f(t)\|_{L_{t}^{\frac{1}{\epsilon'}}}\lesssim 1$, then in time average sense, $\gamma_{p}f(t)\lesssim t^{-\epsilon'}$, thus  in time average sense, $f(t)\lesssim t^{-3(\frac{1}{2}-\frac{1}{p})-\epsilon+\epsilon'}$.
\end{rem}

We now turn to the proof of Lemma \ref{lem: ardis}, which is merely a combination of  Minkowski inequality and some standard large deviation estimate for Gaussian. 
\begin{proof}

It is enough to prove, for $\rho$ large and for $u_{0}^{\omega}=\sum_{n}g_{n}(\omega)\phi_{n}(x)u_{0}(x)$ that 
\begin{equation}\label{eq: ldpwork}
\|\gamma_{p,\epsilon_{1}}(t)e^{it\Delta}u_{0}^{\omega}\|_{L_{\omega}^{\rho}L_{t}^{\frac{1}{\epsilon_{2}}}L_{x}^{p}}\lesssim \sqrt{\rho}\|u_{0}\|_{L^2_x}.
\end{equation}
Then the desired $A$-certain claim comes from usual Chebyshev inequality.
Recall one has, see for example, Lemma 3.1 in \cite{burq2008random},
\begin{equation}\label{eq: ldpstandard}
\|\sum_{n}c_{n}g_{n}(\omega)(\omega)\|_{L_{\omega}^{\rho}}\lesssim \rho^{1/2}(\sum_{n}c_{n}^2)^{1/2}.
\end{equation}
 
Thus, by Minkowski's inequality, we have
\begin{equation}\label{eq: etr1}
\begin{aligned}
&\|\gamma_{p,\epsilon_{1}}(t)e^{it\Delta}u_{0}^{\omega}\|_{L_{\omega}^{\rho}L_{t}^{\frac{1}{\epsilon_{2}}}L_{x}^{p}}\\
\lesssim &\|\|\sum_{n}\gamma_{p,\epsilon}(t)e^{it\Delta}\phi_{n}u_{0}g_{n}(\omega)\|_{L_{\omega}^{\rho}}\|_{L_{t}^{\frac{1}{\epsilon_{2}}}L_{x}^{p}}\\
\lesssim &\rho^{1/2}\|(\sum_{n}|\gamma_{p,\epsilon_{1}}(t)e^{it\Delta}(\phi_{n}u_{0})|^{2})^{1/2}\|_{L_{t}^{\frac{1}{\epsilon_{2}}}L_{x}^{p}},
\end{aligned}
\end{equation}
where in the last step, we apply \eqref{eq: ldpstandard}.
By applying Minkowski inequality again, we have 
\begin{equation}\label{eq: etr2}
\begin{aligned}
&\|(\sum_{n}|\gamma_{p,\epsilon_{1}}(t)e^{it\Delta}(\phi_{n}u_{0})|^{2})^{1/2}\|_{L_{t}^{\frac{1}{\epsilon_{2}}}L_{x}^{p}}\\
\leq &(\sum_{n}\||\gamma_{p,\epsilon_{1}}(t)e^{it\Delta}(\phi_{n}u_{0})|\|_{L_{t}^{\frac{1}{\epsilon_{2}}}L_{x}^{p}}\|^{2})^{1/2}\\
\lesssim &(\sum_{n}\|\phi_{n}u_{0}\|_{L_{x}^{p'}}^{2})^{1/2}
\\
\lesssim &(\|\sum_{n}\phi_{n}u_{0}\|_{L_{x}^{2}}^{2})^{1/2},
\end{aligned}
\end{equation}
where in the second inequality we use the Remark \ref{remweight} and dispersive estimate.

Estimates \eqref{eq: etr1}, \eqref{eq: etr2} gives \eqref{eq: ldpwork}, and the Lemma \ref{lem: ardis} thus follows.
\end{proof}

We also note that $A$-certainly 
\begin{equation}\label{eq: tr}
\|u_{0}^{\omega}\|_{H^{1}}\leq A.
\end{equation}

Thus, we have that $A$-certainly, for some $\beta_{0}>0$,
\begin{equation}\label{eq: scatteringnorm}
\begin{aligned}
&\|u^{\omega}\|_{L_{t,x}^{5}}^{5}\leq A^{\beta_{0}}, \quad \|u^{\omega}\|_{L_{t}^{\infty}H^{1}}\leq A,\\
&\|u_{nl}^{\omega}\|_{L_{t,x}^{5}}^{5}\leq A^{\beta_{0}}, \quad \|u_{nl}^{\omega}\|_{L_{t}^{\infty}H^{1}}\leq A,
\end{aligned}
\end{equation}
where  $u_{nl}^{\omega}=u^{\omega}-e^{it\Delta}u_{0}^{\omega}.$

We will later only relies on $\omega$-wise estimate \eqref{eq: ardis}, \eqref{eq: ardis2},\eqref{eq: scatteringnorm}. For notation's convenience, we fix $\omega$ and omit this $\omega$, we denote $v=e^{it\Delta}u_{0}^{\omega}$ and $w=u_{nl}^{\omega}$. We write down the Duhamel formula,
\begin{equation}\label{eq: secduhamel}
\begin{aligned}
w(t)=-i\int_{0}^{t}e^{i(t-s)\Delta}|v(s)|^{2}v(s)ds-i\int_{0}^{t}e^{i(t-s)\Delta}(|w+v|^{2}(w+v)-|v|^{2}v)ds.
\end{aligned}
\end{equation}
At this time, we  focus on the bootstrap estimate for 
\begin{equation}
A(t):=\sup_{s\leq t}s^{-3/2}\|w(s)\|_{L_{x}^{\infty}}.
\end{equation}

We explain the heuristic why such a generalization will work. 
If $v$ satisfies the exact estimates for some linear solution $e^{it\Delta}f$, $f\in H^{1}\cap L^{1}$, then this is exactly the same proof (for the same problem) as Theorem \ref{thm1}. Here, however, all $v$ involved in \eqref{eq: secduhamel} are in the integral. Thus, for our problem, it is enough for $v$ to satisfy the estimates $e^{it\Delta}f$ in the  time average sense, see Remark \ref{remweight}. There is some $t^{-\epsilon}$ loss but it does not matter since we are never in a critical situation in our setting. We remark that  it remains a very interesting problem to understand the long time dynamic for randomized in physical space $H^{1/2}$ initial data from a quantitative view point.

We carry out the a priori estimate for $w$ and $A(t)$, and later we will also need a parameter $M$.
We write \eqref{eq: secduhamel} as 
\begin{equation}
w=G_{1}+G_{2}+G_{3}+G_{4}
\end{equation}
where\footnote{Here we are slightly abusing  the $O$ notation. $O(fg)$ means the term can be estimated by $fg$, $\bar{f}g$, $f\bar{g}$, $\bar{f}\bar{g}$, and similarly for $O(fgh)$.}
\begin{equation}
\begin{aligned}
G_{1}=-i\int_{0}^{t}e^{i(t-s)\Delta}|v(s)|^{2}v(s)ds,\\
G_{2}=-i\int_{0}^{t}e^{i(t-s)\Delta}|w|^{2}wds,\\
G_{3}=i\int_{0}^{t}e^{i(t-s)\Delta}(O(w^{2}v(s)))ds,\\
G_{4}=i\int_{0}^{t}e^{i(t-s)\Delta}(O(wv^{2}(s)))ds.
\end{aligned}
\end{equation}
The $G_{1}$ part will be addressed in Lemma \ref{lem: g1}.
 The $G_{2}$ can be estimated exactly as in the proof of Theorem \ref{thm1}, i.e. Lemma \ref{lem: thm1local},
 \ref{lem: thm1local2}, \ref{lem: eapriori}.
But if one thinks more carefully, one sees that $G_{3}$ can also be estimated exactly as in the proof of Theorem \ref{thm1}. The only different part may appear in the estimate for $\int_{t-M}^{t}O(w^{2}v)$, which involves the estimate for $\nabla (w^{2}v)$, but in this part, there is still one free $w$ left for us to apply the bootstrap control of $A(t)$, thus one can still estimate it similarly. 
For the $G_{4}$ part, as discussed above, one only needs to estimate the part
\begin{equation}
\int_{t-M}^{t}e^{i(t-s)\Delta}(O(wv^{2}(s)))ds, t\geq 2M,	
\end{equation}
and more precisely, one only needs to handle the part,
\begin{equation}\label{eq: g4e}
\int_{t-M}^{t}(t-s)^{-(1/2+)}\|\nabla w\|_{L^2_x}\|v^{2}(s)\|_{L_{x}^{6}}ds,
\end{equation}
and the other parts follow the same estimate as \eqref{eq: controlforf3}.
The estimate for \eqref{eq: g4e} will be treated in Lemma \ref{lem: g2}.

For $G_{1}$, one has 
\begin{lem}\label{lem: g1}
\begin{equation}
G_{1}(t)\lesssim A^{3}t^{-3/2}.
\end{equation}
\end{lem}
\begin{proof}
One again splits the integral $\int_{0}^{t}=\int_{0}^{t/2}+\int_{t/2}^{t}$.
For the first part, one estimates
\begin{equation}
\big\|\int_{0}^{t/2}e^{i(t-s)\Delta}|v|^{2}v(s)ds\big\|_{L_{x}^{\infty}}
\lesssim t^{-3/2}\int_{0}^{\infty}\|v(s)\|^{3}_{L_{x}^{3}}ds.
\end{equation}

Note that locally, $\|v\|_{L_x^{3}}\leq \|v\|_{H^{1}}\leq A$, and for  $t$ large, by choosing $\epsilon_{1},\epsilon_{2}$ small, we have, thanks to \eqref{eq: ardis}
\begin{equation}
\int_{1}^{t}\gamma_{3,\epsilon_{1}}^{-3}(s)\|\gamma_{3,\epsilon_{1}}(s)v(s)\|_{L_{x}^{3}}^{3}ds
\lesssim \|\gamma_{3,\epsilon_{1}}^{-3}\|_{L_{t}^{\frac{1}{1-3\epsilon_{2}}}}\|\gamma_{3,\epsilon_{1}}(s)v(s)\|_{L_{t}^{1/\epsilon_{2}}L_{x}^{3}}\lesssim A^{3}.
\end{equation}
This handles the $\int_{0}^{t/2}$ part.
For the second part $\int_{t/2}^{t}$, when $t$ is small, one estimates as 
\begin{equation}
\begin{aligned}
&\big\| \int_{t/2}^{t}e^{i(t-s)\Delta}|v(s)|^{2}v(s)ds\big\|_{L_{x}^{\infty}} \\
&\int_{t/2}^{t}\|e^{i(t-s)\Delta}|v(s)|^{2}v(s)\|_{L_{x}^{\infty}}ds \\
&\lesssim \int_{t/2}^{t}\|e^{i(t-s)}\Delta |v|^{2}v\|_{W_x^{1,3+}}ds\\
\lesssim &\int_{t/2}^{t}(t-s)^{-(\frac{1}{2}+)}\|v\|_{H^{1}}\|v\|_{L_{x}^{12-}}\|v\|_{L_{x}^{12-}}ds\\
\lesssim &t^{1/2-}A\int_{t/2}^{t}\gamma_{12-,\epsilon_{1}}^{-2}\|\gamma_{12-,\epsilon_{1}}v\|_{L_{t}^{1/\epsilon_{1}}L_{x}^{12-}}^{2}ds\\
\lesssim &t^{1/2-}(1+t)^{-(\frac{5}{2}-)}.
\end{aligned}
\end{equation}
\end{proof}

For the estimate for \eqref{eq: g4e}, one has 
\begin{lem}\label{lem: g2}
For $M$ large,  $t\geq 2M$, one has 
\begin{equation}
\int_{t-M}^{t}	(t-s)^{-(\frac{1}{2}+)}\|w\|_{H^{1}}\|v^{2}\|_{L^{6}_x}ds\lesssim M^{1/2+}A^{2}(1+t)^{-(5/2-)}.
\end{equation}
\end{lem}
\begin{proof}
This is the parallel computation as in the last part of the proof of Lemma \ref{lem: g1}. Note that 
$\|f^{2}\|_{L_x^{6}}\leq \|f\|_{L_x^{12}}^{2}$,
and we have
\begin{equation}
\begin{aligned}
&\int_{t-M}^{t}	(t-s)^{-(\frac{1}{2}+)}\|w\|_{H^{1}}\|v^{2}\|_{L_x^{6}}ds\\
&\int^{t}_{t-M}(t-s)^{-(\frac{1}{2}+)}\|w\|_{H^{1}}\|v\|_{L_{x}^{12}}\|v\|_{L_{x}^{12}}ds\\
&\lesssim M^{1/2-}A\int_{t/2}^{t}\gamma_{12-,\epsilon_{1}}^{-2}\|\gamma_{12-,\epsilon_{1}}v\|_{L_{t}^{1/\epsilon_{1}}L_{x}^{12-}}^{2}ds\\
&\lesssim M^{1/2+}A^{2}(1+t)^{-(5/2-)}.
\end{aligned}
\end{equation}
\end{proof}

To summarize, similar to \eqref{eq: firsta}, with also Lemma \ref{lem: g1}, \ref{lem: g2}, choosing $M\sim A^{4}$, we have
\begin{equation}\label{eq: randomatlong}
\begin{aligned}
A(t)\leq CA^{3}+CM^{1/2+}A^{2}(1+t)^{-(5/2-)}t^{3/2}
CA^{\beta_{2}}+CA^{\beta_{3}}(\int_{0}^{t}A^{4+}(\|w\|_{L_{x}^{6-}}+\|v\|_{L_{x}^{6-}})^{4+})^{1/4}.
\end{aligned}
\end{equation}
Note that the first two terms in \eqref{eq: randomatlong} come from Lemma \ref{lem: g1}, \ref{lem: g2}.
Applying Gronwall's inequality, the desired estimate \eqref{eq: thm1ran} follows.

We briefly mention some technical point if one wants to further generalize Theorem \ref{thm2}, \ref{thm3}.
One again applies the pseudo-conformal energy trick to prove
\begin{equation}
\|u(t)\|_{L_{x}^{6}}\lesssim t^{-1},	
\end{equation}
and we only need this estimate for $t$ large.

It should be noted in the proof of Theorem \ref{thm2}, \ref{thm3}, it is enough to have
\begin{equation}\label{eq: pd}
\|u(t)\|_{L_{x}^{6}}\lesssim t^{-(1-)},
\end{equation}
for $t$ large, and one only needs the decay estimate \eqref{eq: pd} to hold in time average sense. Thus, we again split $u=v+w$, and \eqref{eq: ardis} will replace estimate \eqref{eq: pd} for $v$, and $w=u-v$ also a priori satisfies  \eqref{eq: pd} in time average sense, and this is enough.
One will also need a $H^{s}$ version of \eqref{eq: tr}. We will prove $A$-certainly,
\begin{equation}
\|\sum_{n}g_{n}(\omega)\phi_{n}u_{0}\|_{H^{s}}\leq A.
\end{equation}
It is enough to prove the following deterministic inequality, 
\begin{equation}
\|f\|_{H^{s}}^{2}\lesssim \sum_{n}\|\phi_{n}f\|_{H^{s}}^{2}.	
\end{equation}
It is enough to prove for all $n$, 
\begin{equation}
\|\phi_{n}\langle \nabla \rangle^{s}f\|_{L^2_x}^{2}\lesssim \sum_{|n-n'|\leq 10}\sum_{n'}\|\phi_{n'}f\|_{H^{s}}^{2},
\end{equation}
which is obvious.

\section{Further remarks}\label{6}
In this section, we make a few more remarks on the decay results for other well known  dispersive equations. We will list some results that can be obtained by using our methods with suitable modifications. We leave the proofs for interested readers. We emphasize that these decay results are based on corresponding scattering results and we will provide the appropriate references. Moreover, we remark here that the implicit constants appearing  are dependent of the size of initial data and the size of scattering norm (for some models, the scattering norm can be expressed by a function of the size of initial data according to existing results).\vspace{3mm}

1. \emph{The 3D, energy critical defocusing NLS case}\vspace{3mm}

The 3D energy critical NLS is an important  and well studied model in the area of dispersive equations. We refer to the milestone work \cite{colliander2008global} for the global well-posedness and scattering result for this model, (see also \cite{bourgain1999global,grillakis2000nonlinear} for the radial case). We write  the Schr\"odinger initial value problem as follows,
\begin{equation}\label{maineq1}
    (i\partial_t+\Delta_{\mathbb{R}^3})u=u|u|^{4}, \quad u(0,x)=u_0(x)\in L^1 \cap H^{1+}(\mathbb{R}^3).
\end{equation} 
Assume $\phi$ is a solution to \eqref{maineq1}. We expect that one could show 
\begin{equation}
   \|\phi(t)\|_{L_{x}^{\infty}} \leq C(u_0) |t|^{-\frac{3}{2}}.
\end{equation}
Moreover, the constant dependence is triple exponential of the initial data since the scattering norm is double exponential and we have a Gronwall's inequality to use. This result obviously improves Theorem 1.1. in \cite{fan2021decay} since the regularity requirement is much lower (from $H^3$ to $H^{1+}$). We also note that one may also consider the higher dimensional case, see \cite{ryckman2007global,visan2007defocusing} for the global results. \vspace{3mm}

We \textbf{emphasize} that, if one only considers $L^p_x$-decay ($p<\infty$), in the sense of showing 
\begin{equation}
   \|\phi(t)\|_{L_{x}^{p}} \leq C(u_0) |t|^{-3(\frac{1}{2}-\frac{1}{p})},
\end{equation}
then assuming $H^1$-regularity instead of $H^{1+}$-regularity for the initial data is already enough, which is more natural. This $\epsilon$-requirement is caused by a log-divergence problem when one considers the $L^{\infty}_x$-decay case. \vspace{3mm}

2. \emph{The energy supercritical NLS case.}\vspace{3mm}

We note that this method can be modified to handle the energy supercritical case once the scattering result is known. For this case, we focus on a typical model as an example: 4D cubic defocusing NLS. The model is $\dot{H}^{\frac{3}{2}}$ critical, which is above the energy critical level. See \cite{dodson2017defocusing} and the references therein for the corresponding scattering result and background. We consider the Schr\"odinger initial value problem,
\begin{equation}\label{maineq2}
    (i\partial_t+\Delta_{\mathbb{R}^4})u=u|u|^{4}, \quad u(0,x)=u_0(x)\in L^1 \cap H^{\frac{3}{2}+}(\mathbb{R}^4).
\end{equation} 
Assume  $\phi$ is solution to  \eqref{maineq2} and the a-priori bound $\limsup_{t\in I} \|u\|_{H^{\frac{3}{2}+}}=M_1 < \infty$, where $I$ is the maximal interval of existence. We expect that one could show \begin{equation}
   \|\phi(t)\|_{L_{x}^{\infty}} \leq C(u_0) t^{-2}.
\end{equation}
Here the constant depends on the size of the initial data, $M_1$ and the scattering norm. 

As a comparison, we also note that, for certain defocusing supercritical NLS problems, one has blow-up type results, see recent result of Merle-Raphael-Rodnianski-Szeftel \cite{merle2022blow} and the references therein for more information.) 
\vspace{3mm}

3. \emph{The fourth-order NLS case.}\vspace{3mm}

There are many different specific fourth-order NLS models. We consider a typical case as an example: the cubic fourth-order Schr\"odinger equation (4NLS) on $\mathbb{R}^d$ ($5 \leq d \leq 8$). We refer to \cite{Pau1} for the corresponding global result and the references therein for the background. We consider the Schr\"odinger initial value problem, for ($5 \leq d \leq 8$),
\begin{equation}\label{maineq3}
    (i\partial_t+(\Delta_{\mathbb{R}^d})^{2})u=u|u|^{p}, \quad u(0,x)=u_0(x) \in L^1 \cap H^{2+}(\mathbb{R}^d).
\end{equation} 
Assume  $\phi$ is solution to \eqref{maineq3}. We expect that one could show
\begin{equation}
   \|\phi(t)\|_{L_{x}^{\infty}} \leq C(u_0) |t|^{-\frac{d}{4}}.
\end{equation}
Here the constant depends on the size of the initial data and the scattering norm. This result improves \cite{yu2022decay} in the sense of regularity requirement. \vspace{3mm}

4. \emph{The fractional NLS case.}\vspace{3mm}

We also consider a typical model and refer to \cite{guo2018energy} for the scattering result. We consider the fractional Schr\"odinger initial value problem  for $d\geq 2$ and $\frac{d}{2d-1}<\alpha<1$,
\begin{equation}\label{maineq4}
    (i\partial_t+(\Delta_{\mathbb{R}^d})^{\alpha})u=u|u|^{\frac{4\alpha}{d-2\alpha}}, \quad u(0,x)=u_0(x) \in L^1 \cap H^{\alpha+}.
\end{equation} 
Assume  $\phi$ is solution to \eqref{maineq4}. We expect that one could show
\begin{equation}
   \|\phi(t)\|_{L_{x}^{\infty}} \leq C(u_0) |t|^{-\frac{d}{2}}.
\end{equation}
Here the constant depends on the size of the initial data and the scattering norm. One may also consider other fractional NLS cases, see \cite{FLS1,FLS2,LiebYau1} and the references therein.\vspace{3mm}

5. \emph{Some  other cases.}\vspace{3mm}

There are  more  models one may consider for an analysis similar to the one we conducted above: cubic-quintic NLS (see \cite{murphy2021threshold,killip2021scattering} and the references therein), inhomogeneous NLS (see \cite{miao2019scattering} and the references therein), NLS on waveguides (see \cite{HP,IPRT3,Z1} for examples), NLS with a partial harmonic potential (this case is similar to the waveguide case, see \cite{antonelli2015scattering,cheng2021scattering,hani2016asymptotic}), Schr\"odinger resonant systems (see \cite{CGZ,yang2018global}), nonlinear wave equations (see \cite{tao2006nonlinear}),  Klein-Gordon equation (see \cite{tao2006nonlinear}) and NLS with a nice potential such that the dispersive estimate and the scattering hold (see \cite{killip2015energy} for an example and the references therein).

\bibliographystyle{plain}
\bibliography{BG.bib}

\end{document}